\documentclass{article}
\usepackage{fullpage}

\usepackage{graphicx} 

\usepackage[a4paper, total={5.5in, 8.5in}]{geometry}

\usepackage{amsmath,amsthm,amssymb}
\usepackage[T1]{fontenc} 
\usepackage[utf8]{inputenc} 
\usepackage{graphicx}
\usepackage{scrextend}
\usepackage{xcolor}
\usepackage{mathtools}

\usepackage{thm-restate} 

\usepackage{enumitem} 
\makeatletter
\newcommand{\customitem}[1]{%
\item[#1]\protected@edef\@currentlabel{#1}%
}
\makeatother

\usepackage[hidelinks]{hyperref} 
\usepackage[noabbrev,capitalise]{cleveref}

\graphicspath{figures/}

\newcommand{\N}{\mathbb{N}}

\newcommand{\R}{\mathbb{R}}
\newcommand{\pos}{\operatorname{pos}}

\newcommand{\col}{\operatorname{col}}

\newtheorem{theorem}{Theorem}
\newtheorem{lemma}[theorem]{Lemma}

\newtheorem{problem}[theorem]{Problem}

\newtheorem{proposition}{Proposition}

\newcommand{\sgn}{\operatorname{sgn}}

\newcommand{\conv}{\operatorname{conv}}

\begin{document}

\title{Blow-ups of order types of positive density}
\author{Ruy Fabila-Monroy\thanks{Departamento de Matem\'aticas, Cinvestav, MEXICO. \texttt{ruyfabila@math.cinvestav.edu.mx}} \thanks{Partially supported by SECIHTI grant  CBF-2025-I-2712.} \and
Benedikt Hahn\thanks{Graz University of Technology, Austria. \texttt{benedikt.hahn@tugraz.at}. This research was funded in part by the Austrian Science Fund (FWF) 10.55776/DOC183.}  \and
Jesús Leaños\thanks{Unidad Académica de Matemáticas, Universidad Autónoma de Zacatecas, México \texttt{jleanos@matematicas.reduaz.mx}} \footnotemark[2]}
\maketitle
\begin{abstract}
Order types are an equivalence relation
between point configurations that capture their
combinatorial and convexity properties.
Let $P$ be a $\kappa$-colored sequence of $n \ge d+1$ points in general position in $\mathbb{R}^d$.
Let $\rho$ be a $\kappa$-colored order type on $k \le d+1$ points that has positive density
on $P$; that is, for some constant $\delta >0$, there are $\delta \cdot \binom{n}{k}$ $k$-point subsequences of
$P$ that have the same order type
as $\rho$ and the same color pattern. In this paper we show that there exists a constant $c >0$ (depending only on $d, \delta$, $k$ and $\kappa$) and disjoint subsets $X_1,\dots,X_k$ of $P$, each with at least $c \cdot n$ points, such that for every choice of $k$ points  $x_i \in X_i$, $(x_1,\dots,x_k)$ has the same order type and  color pattern as $\rho$.
\end{abstract}

\section{Introduction}


A set of points in $\mathbb{R}^d$ is in \emph{general position} if no $k\le d+1$ of them lie in a common $(k-2)$-flat.
Unless stated otherwise, all point sets in this paper are in general position.
Let $P$ and $Q$ be two sets of $n$ points in the plane.
Let $[n]:=\{1,\dots,n\}$.
We say that $P$ and $Q$ have the same \emph{order type}
if there exist bijective functions (called \emph{labellings}),  $\ell_P:[n] \to P$ and $\ell_Q:[n] \to Q$,
such that for every $1 \le i < j < k \le n$ 
either both
\[(\ell_P(i),\ell_P(j), \ell_P(k)) \textrm{ and }   (\ell_Q(i),\ell_Q(j), \ell_Q(k))\]
define a left  turn, or both define a right turn.
Order types can be generalized to $\mathbb{R}^d$, where one now
considers  $(d+1)$-tuples of points. The orientation of a $(d+1)$-tuple of points
$(p_1,\dots,p_{d+1})$ in $\mathbb{R}^d$ is defined by
\[
\chi\left (p_1,\dots,p_{d+1} \right )
\coloneqq\sgn
\det\!\begin{pmatrix}
1 & p_1(1)  & \cdots & p_1(d)\\
1 & p_2(1)  & \cdots & p_2(d)\\
\vdots & \vdots & \ddots & \vdots\\
1 & p_{d+1}(1)  & \cdots & p_{d+1}(d)\\
\end{pmatrix}.
\]
Note that if $\{p_1,\dots,p_{d+1}\}$ is in general position, then $\chi(p_1,\dots,p_{d+1}) \in \{+, -\}$.

The notion of order types was introduced by Goodman and Pollack~\cite{m_sorting}. They capture the combinatorial and convexity properties
of point configurations. For example, in the plane, there are only two different non-equivalent configurations of  four points in general position:
either they form a convex quadrilateral or a triangle with a point in the interior.

This observation provides an alternative formulation of the following  well known problem
in Combinatorial Geometry.  A \emph{rectilinear drawing} of graph
$G$ is a drawing of $G$ in the plane in which its vertices are drawn as point sets in general position and its edges
are straight line segments joining the corresponding pair of points. The \emph{rectilinear crossing number} of $G$ is the minimum
number, $\overline{\operatorname{cr}}(G)$ of pairs of edges that cross in a rectilinear drawing of $G$. For complete graphs
we have that four vertices in a rectilinear drawing define a crossing if and only if they define a $4$-tuple in convex position.
Thus, $\overline{\operatorname{cr}}(K_n)$ is equal to the minimum possible number of convex $4$-tuples in any set of $n$ points
in general position. 
Since every 4-tuple of points is either in convex or non-convex position, we may instead \emph{maximize} the number of non-convex $4$-tuples.

Aichholzer, Duque, Fabila-Monroy, García-Quintero and Hidalgo-Toscano~\cite{cr_proy} used various computer heuristics to find rectilinear
drawings of $K_n$ with very few crossings for 
many small values of $n$.
Their drawing of $K_{2643}$ is used as a seed for the best
known rectilinear drawings of $K_n$ for arbitrarily large values of $n$. The drawings found in \cite{cr_proy}
have the following property. Each of these drawings has four subsets $V_1,V_2,V_3$ and $V_4$ of vertices,
each of about $n/4$ points such that whenever one picks points $x_1 \in V_1$, $x_2 \in V_2$, $x_3 \in V_3$ and $x_4 \in V_4$, they make up a non-convex $4$-tuple $(x_1,x_2,x_3,x_4)$.

Inspired by this empirical observation, in this paper we prove a structural result on point configurations with many
occurrences of a given order type on $k$ points (Theorem~\ref{thm:main_label}). The full strength of our result uses labelled order types and colored points. From that result we  have the following implication that explains this empirical observation.
\begin{proposition}\label{prop:main_nolabel}
 Let $P$ and $\rho$ be sets of $n$ and $k$ points in general position
in $\mathbb{R}^d$, respectively. Suppose for some constant $\delta>0$ there exist $\delta \cdot \binom{n}{k}$ subsets
of $k$ elements of $P$ with the same order type as $\rho$.
Then there exists a constant $\tilde{c}=\tilde{c}(d,\delta,k)>0$ and disjoint subsets
 $X_1,\dots,X_k$ of $P$, each with at least $\tilde{c} \cdot n$ points such that for every choice of $x_1 \in X_1,\dots, x_k \in X_k$, the point set
 $\{x_1,\dots,x_k\}$ has the same order type as $\rho$.
\end{proposition}

We point out the importance that geometry plays in our result. Consider the following combinatorial reformulation of the problem.
Let $H$ be a $k$-uniform hypergraph with vertex set equal to $P$, and where a subset of $k$ vertices of $H$ is
an edge if and only if the corresponding points have the same order type as $\rho$. $H$ is dense, and the $X_i$ in Proposition~\ref{prop:main_nolabel}
define a complete $k$-partite hypergraph. Now consider the random $k$-uniform hypergraph $J$ with vertices
on $P$ that results from adding every edge with probability $\delta$. The expected number of edges of $J$ is equal
to $\delta \cdot \binom{n}{k}.$ The probability
that some subsets of vertices $X_1,\dots,X_k$, each of size $t$, define a complete $k$-partite hypergraph is equal to
$\delta^{t^k}$. Therefore, the probability that there exists such a choice of subsets is at most
\[\binom{n}{t}^k \cdot \delta^{t^k} \le n^{tk}\cdot \delta^{t^k};\]
which tends to zero for $t\ge k\log n$.

We now give the definitions needed to state Theorem~\ref{thm:main_label}.
In the remaining paper, we regard any collection of points $P$ in $\mathbb{R}^d$ as a finite sequence (instead of a set) and denote the $i$th entry of a sequence by $P(i)$.
In this context two sequences $P$ and $Q$ of $n$ points in $\mathbb{R}^d$ have the same \emph{labelled order type}
if for every tuple $1\le i_1 < \cdots < i_{d+1} \le n$, 
\[\chi(P(i_1),\dots,P(i_{d+1}))=\chi(Q(i_1),\dots,Q(i_{d+1})).\]
We denote this with
$P \simeq Q$. Note that when $P$ and $Q$ have different labelled order types, they might still have the same (unlabelled) order type,
when regarded as sets.

We assume that the points in any sequence $P$ are $\kappa$-colored by a function $\col : P \to [\kappa]$. 
Here, $\kappa$ is an integer that depends on the context.
If no coloring is mentioned, we assume $\kappa=1$.
We say that two sequences $\rho,\tau$ of $k$ points
are \emph{copies (of one another)}, if $\tau \simeq \rho$  and
$(\col(\tau(1)),\dots,\col(\tau(k)))=(\col(\rho(1)),\dots,\col(\rho(k)))$.
In an abuse of notation, we allow $\col$ to assign different colors to the same point $x \in \R^d$, depending on the sequence that $x$ was chosen from.
When a subsequence of $P$ is a copy of $\rho$, we write \[\rho \preccurlyeq P.\]

We say that $\rho$ has \emph{positive density} or that it has \emph{density $\delta$} on $P$ if there exists a constant
$\delta >0$ such that at least $\delta \binom{n}{k}$ of the $k$-element subsequences of $P$ are copies of $\rho$.\footnote{In general, we say that $Y\subset X$ has \emph{positive density} in $X$, if there
exists a constant $\delta >0$ such that $|Y| \ge \delta \cdot |X|$.}
Let $X_1,\dots,X_k$ be disjoint subsequences of $P$.
We say that
$(X_1,\dots,X_k)$ has \emph{$\rho$-type transversals} if every one of its transversals $(x_1, \ldots, x_k)$ with $x_i \in X_i$
is a copy of $\rho$.
If in addition they are also subsequences of $P$, then we say that
$(X_1,\dots,X_k)$ has \emph{$\rho$-type transversals with respect to}  $P$.
If $\rho$ is not specified, then we say that $(X_1,\dots,X_k)$ has \emph{same-type transversals}.

We are ready to state our main result.
\footnote{After finishing this preprint, we learned that \cref{thm:main_label}
follows directly from previous work of Fox, Pach and Suk~\cite[Corollary~1.2]{semigraphs}.
The underlying regularity method originates with Fox, Gromov, Lafforgue, Naor and
Pach~\cite{gromov}. Our proof is of a different, more geometric nature, relying only on
the second selection lemma and the signed central projection of \cref{sec:sgn_proj}.}
\begin{restatable}{theorem}{mainlabel}\label{thm:main_label}
Let $P$ and $\rho$ be two sequences of $n$ and $k$ points in general position
in $\mathbb{R}^d$, and let $\col$ be a $\kappa$-coloring of both sequences. 
Suppose that  $\rho$ has density $\delta>0$ in $P$.
Then, there exists a constant $c=c(d,\delta,k, \kappa)>0$ and disjoint subsequences
 $X_1,\dots,X_k$ of $P$, each with at least $c \cdot n$ points such that $(X_1,\dots,X_k)$ has
 $\rho$-type transversals with respect to $P$.
\end{restatable}
Note that since there are at most $k!$ ways to label a given set of $k$ points,
Theorem~\ref{thm:main_label} readily implies Proposition~\ref{prop:main_nolabel} and so, $\tilde{c}(d,\delta,k) \geq c(d,\delta/k!,k,1)$.

\subsection{Related work}

In the literature there are various results about Euclidean point sets  similar to Proposition~\ref{prop:main_nolabel}; in the sense that
they guarantee the existence of special substructures of linear size of the original point set.
As we will see, many of these results can be derived from \cref{thm:main_label}.
One especially closely related result is the same-type lemma by Bárány and Valtr, which guarantees large subsets of any given sets of points with same-type transversals.
\begin{theorem}[Same-type lemma~\cite{positive_es}]
Let $X_1,\dots,X_k$ be disjoint subsets of points in $\mathbb{R}^d$ such that $X_1 \cup \cdots \cup X_k$
is in general position. Then there exists a positive constant $c_{ST}=c_{ST}(d,k)$ and subsets $Y_1 \subset X_1,\dots, Y_k \subset X_k$, such
that $(Y_1,\dots,Y_k)$ has same-type transversals and $|Y_i| \ge c_{ST} |X_i|$ for every $1 \le i \le k$.
\end{theorem}
While the upper bound on $c_{ST}(d,k)$ was originally double exponential, it has subsequently been improved to an exponential \cite{fox_expon_same_type} and then a polynomial \cite{bukh_poly_same_type}.
We can derive the same-type lemma from Theorem~\ref{thm:main_label} with the following trick.

\begin{proof}[Proof from \cref{thm:main_label}]
Replace every point $x_i \in X_i$ with a cluster of $\prod_{j\neq i}|X_j|$ points very close to $x_i$ and assign
color $i$ to these points. Note that  the newly defined sets all have size equal to $\prod_{j=1}^k|X_j|$ .
Afterwards, label the points in any order. Let $P$ be the resulting sequence of $k$-colored points.
Note that $n:=|P|=k\cdot \prod_{j=1}^k|X_j|$.
The number of different labelled order types on $k$ points is a constant that depends only on $k$ and $d$.
Since each color class has $n/k$ elements, a constant fraction of all $k$-element subsequences of $P$ have all $k$ colors.
Therefore, there is an $k$-colored
sequence $\rho$ of $k$ points, with all its points of different color, with positive density $\delta$ in $P$.
Without loss of generality assume that for every $1 \le i \le k$, $\rho(i)$ has color $i$.
Apply Theorem~\ref{thm:main_label} to $P$  to obtain $X_1',\dots,X_k'$ disjoint subsets each with $c \cdot n$ points, for some
constant $c=c(d,\delta,k,k)>0$, such
that $(X_1',\dots,X_k')$ has $\rho$-type transversals with respect to $P$.
Note that every element of $X_i'$ has color $\col(\rho(i))=i$; therefore every point in $X_i'$ belongs
to the cluster of a point in $X_i$.
For every $1 \le i \le k$, let $Y_i \subset X_i$
be the set of elements $x_i \in X_i$ such that at least one of the points in the cluster of $x_i$ is in $X_i'$.
We have that
\[|Y_i|\ge \frac{|X_i'|}{\prod_{j\neq i}|X_j|}\ge \frac{c \cdot k \cdot \prod_{j=1}^k|X_j|}{\prod_{j\neq i}|X_j|} =c \cdot k \cdot |X_i|.\]
Moreover, if the points in the cluster of $x_i$ are placed sufficiently close to $x_i$, then every transversal to $(Y_1,\dots,Y_k)$ has the same order type as $\rho$.
\end{proof}

Bárány and Valtr used their result to show the following ``positive fraction'' version of the Erd\H{o}s--Szekeres Theorem.
\begin{theorem}
For every integer $k\ge 3$ there exists a constant $c_{ES}=c_{ES}(k)>0$ with the following property.
Every sufficiently large set of $n$ points $P$ in general position in the plane contains $k$ subsets
$Y_1,\dots,Y_k$ with $|Y_i| \ge c_{ES} \cdot n$ (for every $1 \le i \le k$) such that every
transversal to $(Y_1,\dots,Y_k)$ is in convex position.
\end{theorem}
\qed\\
The Erd\H{o}s--Szekeres Theorem~\cite{happyend} states that for every integer $k>0$, there exists a sufficiently large integer $n(k)$ such that
every set of at least $n \ge n(k)$ points in the plane contains $k$ points in convex position. Note that
such a set contains at least
\[\left. \binom{n}{n(k)}\, \middle/ \,\binom{n-k}{n(k)-k} \right.= \frac{1}{\binom{n(k)}{k}} \cdot \binom{n}{k}\]
$k$-tuples in convex position.
Thus, the positive fraction  Erd\H{o}s--Szekeres Theorem directly follows from
Theorem~\ref{thm:main_label}.

We mention a recent result by Fabila and Huemer, similar in spirit to Proposition~\ref{prop:main_nolabel}; in our setting it can be stated as follows.
\begin{theorem}[Depth Caratheódory Theorem~\cite{car_depth}]\label{thm:depth_car}
 Let $P$ be a set of $n$ points in general position in $\mathbb{R}^d$, and let $q$ be a point
 in the interior of at least $\delta \cdot \binom{n}{d+1}$ simplices with vertices on $P$, for some constant $\delta >0$.
 Then, there exist a constant
 $c_{DC}=c_{DC}(d,\delta)>0$ and subsets $X_1,\dots,X_{d+1}$ of $P$, each of at least $c_{DC}\cdot n$ points, such that every transversal to $(X_1,\dots,X_{d+1})$
 defines a simplex that contains $q$ in its interior.
\end{theorem}
This again follows from \cref{thm:main_label}.
\begin{proof}[Proof sketch from \cref{thm:main_label}]
    Let $Q$ be a set of $n$ points sufficiently close to $q$ such that each of the $\delta \binom{n}{d+1}$ simplices that contains $q$ also contains all points of $Q$.
    Let $\Delta$ be such a simplex and $q' \in Q$.
    Let further, $P' \coloneqq P \cup Q$ and $\col : P' \to [2]$ a 2-coloring of $P'$ with $\col(P) = 1$ and $\col(Q) = 2$.
    Then, applying \cref{thm:main_label} to $P'$ and the (unlabelled) order type $\Delta \cup \{q'\}$ implies the result.
\end{proof}

Finally we mention a relationship of our result with a problem on permutations.
In the case of $d=1$ and $\{P(i):1 \le i \le n\}=[n]$, $P$ is simply a permutation
of $[n]$. In this setting, a copy of a labelled order type $\rho$ is an occurrence
of the permutation pattern $\rho$, that is, a subsequence of $P$ with the same relative order as $\rho$.
Given $I,J \subset [n]$, we say that $I<J$, if for every $i \in I$ and $j \in J$, we have
that $i<j$.
Thus, Theorem~\ref{thm:main_label} applies to permutations as well: if a pattern $\rho$ has positive density
$\delta >0$
in a permutation $P$, then there exist pairwise disjoint sets of indices $I_1<\cdots<I_k$,
each of size at least $c \cdot n$ for some constant $c=c(\delta,k)>0$,
such that for every transversal $(i_1,\dots,i_k)$ of $(I_1,\dots,I_k)$
we have that the subsequence
\[(P(i_1),\dots,P(i_k))\] is
an occurrence of $\rho$.
The problem of finding the maximum possible density of a given pattern $\rho$ in a permutation of $[n]$ is known
as \emph{permutation packing}; see Albert, Atkinson, Handley, Holton and Stromquist~\cite{perm_pack}.

\paragraph{Outline.} This paper is organized as follows. In Sections~\ref{sec:trans},~\ref{sec:sgn_proj} and~\ref{sec:pin}, we present the three
main tools we use for the proof of Theorem~\ref{thm:main_label}. We believe that these tools may be
of independent interest.

In Lemma~\ref{lem:transversals} of Section~\ref{sec:trans}, we show
that if there exist linear sized subsets $(X_1,\dots,X_k)$ of $P$ with $\rho$-type transversals such that
a positive fraction of them are also subsequences of $P$, then we can find subsets $Y_i \subset X_i$ of linear
size such that $(Y_1,\dots,Y_k)$ has $\rho$-type transversals with respect to $P$. In particular the problem
is now reduced to finding these $X_i$.
In Theorem~\ref{thm:proj} of Section \ref{sec:sgn_proj}, we show that we can project $P$ from a point $o$ to a hyperplane below $o$, and recover the order type
of a simplex containing $o$ at the cost of adding the information to every point on whether its preimage is below or above $o$.
In Theorem~\ref{thm:pin} of Section~\ref{sec:pin}, we show that we can find a set $Q$ of $\binom{n}{k}$ points in $\mathbb{R}^d$
such that for a constant fraction of the copies of $\rho$  in $P$, we have that every one of its simplices
contains a point of $Q$ in its interior.
These points enable us to do the projection described in Section~\ref{sec:sgn_proj}.

Finally, in Section~\ref{sec:proof} we prove Theorem~\ref{thm:main_label} via a further strengthening of \cref{thm:main_label} (\cref{lem:synced-sets}).

\section{From $\rho$-type transversals to $\rho$-type subsequences of $P$}\label{sec:trans}

Let $k \le n$ be a positive integer. For a given set $X$, let $X^k$ denote the $k$-fold Cartesian product of $X$ with itself; and let
\[
X^{(k)}\coloneqq\{(x_1,\dots,x_k)\in X^k:\ x_i\neq x_j\text{ for }i\neq j\}.
\]
If $X$ is a set of real numbers, then
let \[X_{<}^{(k)}:=\left \{t \in X^{(k)}: 1 \le t(1) < t(2) < \cdots < t(k) \le n \right \}.\]

Let $f$ be a function with domain equal to $Y$ and $\tau$ be a sequence of $k$ elements in $Y$. We define
\[f[\tau]:=(f(\tau(1)),\dots,f(\tau(k))).\]
Note that we can regard a sequence of $n$ elements as a function with domain equal to $[n]$; thus,
if now $\sigma$ is a sequence of $n$ elements and $\tau$ a sequence of $k$ elements in $[n]$,
then \[\sigma[\tau]=(\sigma(\tau(1)),\dots,\sigma(\tau(k))).\]
In what follows, let $P\in (\mathbb{R}^d)^k$. For every $p \in\{P(i):1 \le i \le n\}$,
we define the \emph{position} of $p$ in $P$ as the index $\pos_P(p)$, such that $P(\pos_P(p))=p$.

\begin{lemma}\label{lem:transversals}
 Let $\rho$ be a $k$-point sequence in $\R^d$ and $\delta >0$. Suppose that there exist disjoint subsets $X_1,\cdots,X_k$ of $P$, such that:
 \begin{itemize}
  \item $(X_1,\dots,X_k)$ has $\rho$-type transversals; and
  \item there exist $\delta \cdot \binom{n}{k}$ transversals to $(X_1,\dots,X_k)$ that are subsequences of $P$.
 \end{itemize}
 Then, there exists a constant $\alpha=\alpha(d,\delta,k)>0$ and subsets
 $Y_1 \subset X_1,\dots,Y_k \subset X_k$ such that:
 \begin{itemize}
  \item[$(a)$] each $Y_i$ has at least $\alpha \cdot n$ points; and

  \item[$(b)$] $(Y_1,\dots,Y_k)$ has $\rho$-type transversals with respect to $P$.
 \end{itemize}
\end{lemma}
\begin{proof}
We define the $Y_i$ iteratively.
Let $T_0$ be the set of transversals to $(X_1,\dots, X_k)$ that are subsequences of $P$.
We have that
\[|T_0| \ge \delta \cdot \binom{n}{k}.\]
 We say that a point $p \in X_1$ is \emph{heavy} if there are at least $|T_0|/2n$, transversals $\tau \in T_0$,
 such that $\tau(1)=p$. We also say that $\tau \in T_0$ is \emph{heavy} if $\tau(1)$ is heavy.
 Let \[Y_1\coloneqq\{p \in X_1: p \textrm{ is heavy}\}.\]

 By counting for every non-heavy element of $X_1$, we have that the number of non-heavy transversals in $T_0$ is at most
 \[|X_1|\cdot \frac{|T_0|}{2n} \le \frac{|T_0|}{2}.\] Therefore, the number of heavy transversals in $T_0$ is at least \[\frac{\delta}{2} \cdot \binom{n}{k}.\]
 Since every point in $X_1$ is in at most $\binom{n}{k-1}$ transversals in $T_0$, we have that
 \[|Y_1|  \ge \left. \frac{\delta}{2} \cdot \binom{n}{k}\,\middle/\,  \binom{n}{k-1} \right.  > \frac{\delta}{2^2 k} \cdot n,\]
 for sufficiently large $n$.

 Let $q_1$ be the point of $Y_1$ with  maximum value $\pos_P(q_1)$. Let
 \[T_1\coloneqq\left \{\tau \in T_0: \tau(1)=q_1 \right \}.\]
 Since $q_1$ is heavy we have that
 \[|T_1| \ge \frac{1}{2n}|T_0| \ge \frac{\delta}{2n} \binom{n}{k} \ge \frac{\delta}{2^2k} \cdot \binom{n}{k-1},\]
 for sufficiently large $n$.

 We proceed by induction on $i$. Suppose that for $2 \le i < k$, we have defined:
 \begin{itemize}
  \item sets of points, $Y_1 \subset X_1, \cdots, Y_{i-1} \subset X_{i-1}$;

  \item sets of points $q_1 \in Y_1,\dots, q_{i-1} \in Y_{i-1}$; and

  \item sets of transversals to $(X_1,\dots,X_k)$, $T_{i-1} \subset T_{i-2} \subset \cdots \subset T_0$.
 \end{itemize}
 that satisfy the following,
 for every $1 \le j \le i-1$:
 \begin{itemize}
 \item[$(1)$] for every point $p \in Y_j$, there exist at least $|T_{j-1}|/2n$ transversals in $\tau \in T_{j-1}$
 such that \[\tau(j)=p;\]

 \item[$(2)$] \[|Y_j| \ge \frac{\delta}{2^{2j}k^j} \cdot n;\]

 \item[$(3)$] $q_j$ is the point of $Y_j$ with  maximum value $\pos_P(q_j)$;

 \item[$(4)$] \[T_j\coloneqq\{\tau \in T_0:\tau(l)=q_l \textrm{ for every } 1 \le l \le j\};\]

 \item[$(5)$] \[|T_j| \ge \frac{\delta}{2^{2j}k^j}\binom{n}{k-j}.\]
 \end{itemize}
We now define $Y_i, q_i$ and $T_i$.

We say that a point $p \in X_i$ is \emph{heavy} if there are at least $|T_{i-1}|/2n$, transversals $\tau \in T_{i-1}$,
 such that $\tau(i)=p$. We also say that $\tau \in T_{i-1}$ is \emph{heavy} if $\tau(i)$ is heavy.
 Let \[Y_i\coloneqq\{p \in X_i: p \textrm{ is heavy}\}.\]
 By definition $Y_i$ satisfies $(1)$.

 By counting for every non-heavy element of $X_i$, we have that the number of non-heavy transversals in $T_{i-1}$ is at most
 \[|X_i|\cdot \frac{|T_{i-1}|}{2n} \le \frac{|T_{i-1}|}{2}.\] By $(4)$ the number of heavy transversals in $T_{i-1}$ is at least
 \[\frac{1}{2} \cdot \frac{\delta}{2^{2(i-1)}k^{i-1}}\binom{n}{k-i+1}.\]
 Since every point in $X_{i+1}$ is in at most $\binom{n}{k-i}$ transversals in $T_{i-1}$, we have that
 \[|Y_i|  \ge \left. \frac{1}{2} \cdot \frac{c}{2^{2(i-1)} k^{i-1}} \cdot \binom{n}{k-i+1}\,\middle/\,  \binom{n}{k-i} \right.  >
 \frac{1}{2} \cdot \frac{\delta}{2^{2(i-1)} k^{i-1}}\cdot \frac{n}{2(k-i)} >\frac{\delta}{2^{2i} k^i} \cdot n,\]
 for sufficiently large $n$. Thus, $Y_i$ satisfies $(2)$.

 Let $q_i$ be the point of $Y_i$ with  maximum value $\pos_P(q_i)$. Let
 \[T_i\coloneqq\left \{\tau \in T_{i-1}: \tau(i)=q_i \right \}.\]
 By definition $q_i$ and $T_i$ satisfy $(3)$ and $(4)$, respectively.

 Since $q_i$ is heavy, by $(5)$ we have that
 \[|T_i| \ge \frac{1}{2n}|T_{i-1}| \ge \frac{1}{2n} \cdot \frac{\delta}{2^{2(i-1)}k^{i-1}}\binom{n}{k-i+1}
 \ge \frac{\delta}{2^{2i}k^{i}} \binom{n}{k-i},\]
 for sufficiently large $n$. Thus, $T_i$ satisfies $(5)$.

 Suppose we have defined up to $Y_{k-1}$, $q_{k-1}$ and $T_{k-1}$ that satisfy $(1)-(5)$.
 Let \[Y_k\coloneqq\{p \in X_k:\textrm{ there exists } \tau \in T_{k-1} \textrm{ such that } \tau(k)=p\}.\]
 By $(5)$, we have that
 \[|Y_k|=|T_{k-1}|\ge \frac{\delta}{2^{2(k-1)}k^{k-1}} \cdot n\]
Let $q_k$ be the element of $Y_k$ of maximum value $\pos_P(q_k)$ and let
$T_k\coloneqq\{(q_1,\dots,q_k)\}$.

 Let
 \[\alpha \coloneqq  \frac{\delta}{2^{2(k-1)}k^{k-1}}.\]
 By $(2)$, we have that every $Y_i$ satisfies that
 \[|Y_i|\ge \alpha \cdot n. \]
 This proves $(a)$.

 Since $T_k \subset T_{k-1} \cdots \subset T_0$, we have that
 $(q_1,\dots,q_k) \in T_0$. Therefore, $(q_1,q_2,\dots,q_k)$ is a subsequence of $P$.
 Let $(y_1,\dots,y_k)$ be a transversal to $(Y_1,\dots,Y_k)$. Let $1 \le j \le k$. By $(3)$
 we have that \[\pos_P(y_j) \le \pos_P(q_j).\] Suppose that $j >1$. Since $y_j\in Y_j$, there exists
 a transversal $\tau \in T_{j-1}$, such that $\tau(j)=y_j$. Since $\tau \in T_0$, this implies
 that \[\pos_P(q_{j-1})=\pos_P(\tau(j-1)) < \pos_P(\tau(j)) =\pos_P(y_j).\]
 Therefore,
 \[\pos_P(y_1) < \pos_P(y_2) < \cdots < \pos_P(y_k),\]
 and $(y_1,\dots,y_k)$ is a subsequence of $P$.
 This proves $(b)$.

\end{proof}

\section{Signed central projection}\label{sec:sgn_proj}

Let  $o \in \mathbb{R}^d$, $\Pi$ some oriented hyperplane passing through $o$, and $\Pi'$ another hyperplane parallel to $\Pi$ which we identify with $\R^{d-1}$.
We define the \emph{signed central projection from $o$} for points $q \in \mathbb{R}^d\setminus \Pi$ as follows.
Let $\ell_q$ be the straight line through $o$ and $q$.
The signed central projection from $o$ maps $q$ to the point $\pi_o(q)$
in the intersection $\ell_q \cap \Pi$.
Here, we consider $\pi_o(q)$ as a point in $\R^{d-1}$. We  add a \emph{sign} to $\pi_o(q)$, where we set $\sgn(\pi_o(q))\coloneqq +1$ if $q$ is ``above'' $\Pi$ and $\sgn(\pi_o(q))\coloneqq -1$ if $q$ is ``below'' $\Pi$.
Although the projection depends on $\Pi$, we don't include $\Pi$ in our notation of $\pi_o$, as the specific choice 
is not relevant for our purposes.

See \cref{fig:signed-central-projection} for an illustration.

\begin{figure}[h]
    \centering
    \includegraphics[scale=1]{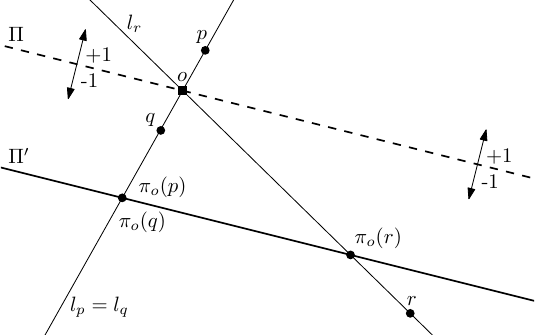}
    \caption{The signed central projection from $o$ with respect to the hyperplanes $\Pi$ and $\Pi'$.
    Because $p,q$ and $o$ are collinear, $p$ and $q$ are sent to the same point $\pi_o(p) = \pi_o(q)$ but with different signs.
    In the remainder of the paper, no such collisions of points in $P$ occur, as $P \cup \{o\}$ will be in general position.}
    \label{fig:signed-central-projection}
\end{figure}

Note that, if $\Pi'$ is the hyperplane given by $x_d=0$, $\Pi$ the hyperplane given by $x_d=1$, and $o$ is the origin, then
\[\pi_o(q)=\left (\frac{q(1)}{q(d)},\dots, \frac{q(d-1)}{q(d)}\right).\]
In what follows we choose $o$ and $\Pi$ so that $P \cup\{o\}$ is in general position, and $\Pi \cap P= \emptyset;$ in this setting $\pi_o$
when restricted to $P$ is injective. We use these assumptions
to derive a formula for the order type of a point set after a signed central projection.

\begin{lemma}\label{lem:proj}
Let $\{o\} \cup P \subseteq \R^d$ be points in general position, $\Pi$ a hyperplane passing through $o$ but disjoint to $P$, $\Pi'$ a hyperplane parallel to $\Pi$, and $\pi_o$ the induced signed central projection.
Then, for any subsequence $P' = (p_1, \ldots, p_d) \in P^{(d)}$ of $P$,
    $$\chi(o,p_1,\ldots, p_d) = \sgn \left (
    (-1)^{d-1}\right ) \cdot
    \left ( \prod_{j=1}^d\sgn ( \pi_o(p_{j}) ) \right ) \cdot \chi \left (\pi_o(P')\right).$$
    In particular, $\pi_o(P)$ is in general position and the labelled order type of $\pi_o(P)$ depends only on the labelled order type of $\{o\}\cup P$ and the signs of the projected points.
\end{lemma}
\begin{proof}
Order types are invariant under rotations, translations and scaling. Thus, we may assume that $o$ is the origin, that $\Pi$ is the hyperplane given by $x_d=0$, and that $\Pi'$ is the hyperplane given by $x_d=1$. Thus, the sign of a projected point is given by the sign of its $d$-th coordinate.

Let $P' = (p_1, \ldots, p_d) \in P^{(d)}$.
We have that
\begin{align*}\chi\left(o,p_1, \ldots, p_d\right)
&=\sgn
\det\!\begin{pmatrix}
1 & 0 & \cdots & 0\\
1 & p_{1}(1) & \cdots & p_{1}(d)\\
\vdots & \vdots & \ddots & \vdots\\
1 & p_{d}(1)  & \cdots & p_{d}(d)\\
\end{pmatrix}\\[1ex]
&=\sgn
\det\!\begin{pmatrix}
 p_{1}(1) & \cdots & p_{1}(d)\\
  \vdots & \ddots & \vdots\\
 p_{d}(1)  & \cdots & p_{d}(d)\\
\end{pmatrix}\\[1ex]
&=\sgn \left (
\frac{1}{\prod_{j} p_{j}(d)} \cdot
\det\!\begin{pmatrix}
\frac{p_{1}(1)}{p_1(d)} &  \cdots & \frac{p_{1}(d-1)}{p_1(d)} & 1\\
\vdots & \ddots & \vdots & \vdots \\
\frac{p_{d}(1)}{p_{d}(d)} &  \cdots & \frac{p_{d}(d-1)}{p_{d}(d)} & 1\\
\end{pmatrix} \right )\\[1ex]
&=\sgn \left (
\frac{(-1)^{d-1}}{\prod_{j} p_{j}(d)} \cdot
\det\!\begin{pmatrix}
1 & \frac{p_{1}(1)}{p_1(d)} &  \cdots & \frac{p_{1}(d-1)}{p_1(d)} \\
\vdots & \vdots & \ddots &  \vdots \\
1 & \frac{p_{d}(1)}{p_{d}(d)} &  \cdots & \frac{p_{d}(d-1)}{p_{d}(d)} \\
\end{pmatrix} \right )\\[1ex]
&=\sgn \left (
(-1)^{d-1}\right ) \cdot
\left ( \prod_{j=1}^d\sgn ( \pi_o(p_{j}) ) \right ) \cdot \chi \left (\pi_o(P')\right),
\end{align*}
as desired.
\end{proof}

%
%

\begin{theorem}\label{thm:proj}
Let $Q:=(q_1,\dots,q_{d+1})$ and $Q':=(q_1',\dots,q_{d+1}')$ be  $(d+1)$ element sequences of points in $\mathbb{R}^d$.
Let $o \in \mathbb{R}^d$ and $\Pi$ a hyperplane, such that: $Q \cup Q' \cup \{o\}$ is in general position;
$Q\cap \Pi=Q'\cap \Pi=\emptyset$; and  $o \in \conv(Q)$.
Suppose that
\begin{itemize}
\item $\pi_o[Q] \simeq \pi_o[Q']$ and
\item  $\sgn[\pi_o[Q]]=\sgn[\pi_o[Q']].$
 \end{itemize}
 Then
 \begin{enumerate}
 \customitem{$(1)$} $o \in \conv(Q') $ and \label{item:proj1}
 \customitem{$(2)$} $Q \simeq Q'.$ \label{item:proj2}
 \end{enumerate}

\end{theorem}
\begin{proof}
Let $1 \le i \le d+1$. We denote by $Q_i$, the tuple $(q_1,\dots,q_{d+1})$ with
$q_i$ omitted. We define $Q_i'$ analogously.
By  \cref{lem:proj}, we have that
$$\chi(o,Q_i)= \sgn \left (
(-1)^{d-1}\right ) \cdot \left ( \prod_{j\neq i}\sgn ( \pi_o(q_{j}) ) \right ) \cdot \chi \left (\pi_o(Q_i)\right)$$
and
$$\chi(o,Q_i')= \sgn \left (
(-1)^{d-1}\right ) \cdot \left ( \prod_{j\neq i}\sgn ( \pi_o(q_{j}') ) \right ) \cdot \chi \left (\pi_o(Q_i')\right).$$
Therefore, by assumption
$$\chi\left (o,Q_i\right)=\chi\left (o,Q_i'\right).$$

Since $o$ is in the interior of $\conv(Q)$, we have
that for every $1 \le i \le d+1$, $o$ and $q_i$ lie on the same side of the halfspace bounded
by the hyperplane spanned by $Q_i$.
Therefore,
\[\chi\left (o,Q_i \right)=\chi\left (q_i,Q_i\right) =(-1)^{i-1}\chi\left( Q \right).\]

Suppose for a contradiction that $o \notin  \conv(Q')$ and let $r$ be an infinite ray
with apex $o$ that intersects the interior of $\conv(Q')$. Note that $r$ intersects
two faces of the simplex spanned by $Q'$. Let $i$ and $j$ be the indices such that
these faces are contained in the hyperplanes $H_i$ and $H_j$ spanned by $Q_i'$ and $Q_j'$, respectively.
Without loss of generality assume that $r$ intersects $H_i$ first, so that $q_i'$ and $o$ lie on opposite sides
of $H_i$, and $o$ and $q_j'$ lie on the same side of $H_j$.

Since $\chi(o,Q_i')=-\chi(q_i',Q_i')$, we have that
\begin{align*}
 (-1)^{i-1}\chi\left (Q \right) & = \chi\left (o,Q_i \right) \\
 &= \chi\left (o,Q_i' \right) \\
 &=-\chi(q_i',Q_i')\\
 &=(-1)^{i}\chi\left (Q' \right);
\end{align*}
which implies that $\chi\left (Q \right) =-\chi\left (Q' \right).$

On the other hand, since $\chi(o,Q_j')=\chi(q_j',Q_j')$, we have that
\begin{align*}
 (-1)^{j-1}\chi\left (Q \right) & = \chi\left (o,Q_j \right) \\
 &= \chi\left (o,Q_j' \right) \\
 &=\chi(q_j',Q_j')\\
 &=(-1)^{j-1}\chi\left (Q' \right);
\end{align*}
which implies that $\chi\left (Q \right) =\chi\left (Q' \right),$
a contradiction. This proves \ref{item:proj1}.

Therefore, for any $1 \le i \le d+1$, we have that
\begin{align*}
 (-1)^{i-1}\chi\left (Q \right) & = \chi\left (o,Q_i \right) \\
 &= \chi\left (o,Q_i' \right) \\
 &=\chi(q_i',Q_i')\\
 &=(-1)^{i-1}\chi\left (Q' \right);
\end{align*}
which implies that $\chi\left (Q \right) =\chi\left (Q'\right)$, proving \ref{item:proj2}.
\end{proof}

\section{Pinning sequences of points}\label{sec:pin}

Alon, Bárány, Füredi and Kleitman proved the following result that allows us to find a point that ``pins'' a large family of simplices with vertices on $P$.
\begin{theorem}[Second selection lemma~\cite{second_selection}]
Let $\mathcal{T}$ be a family of $\alpha \cdot \binom{n}{d+1}$ simplices with vertices on $P$, for some
constant $\delta >0$. Then there exists a point contained in the interior of at least
\[\beta_{SS} \cdot \binom{n}{d+1}\]
simplices of $\mathcal{T}$, for some constant $\beta_{SS} = \beta_{SS}(d,\delta)>0$.
\end{theorem}
\qed

Let $\rho \in \left ( \mathbb{R}^d\right )^{(k)}$ and  $I \subset [k]_{<}^{(d+1)}$;
we say that a set of points
\[Q\coloneqq\left \{q_{t} \in \mathbb{R}^d: t \in I  \right \},\]
indexed by  $I$,
\emph{pins} $\rho$, if every point $q_t \in Q$ is contained in the interior of the simplex with vertices
$\rho[t]$.
%

\begin{theorem}\label{thm:pin}
Let $P$ be a sequence of $n$ points, $\delta >0$ a constant, $I \subseteq [k]_{<}^{(d+1)}$, and $\mathcal{T}$ a family of $\delta \cdot \binom{n}{k}$ subsequences of
$k$ points of $P$, all copies of some sequence $\rho$ on $k$ points. Then there exists:
 \begin{itemize}
  \item a subfamily $\mathcal{T}' \subset \mathcal{T}$ with $\beta \binom{n}{k}$ elements, for some constant
  $\beta=\beta(d,\delta,k)$; and
  \item a point set \[Q\coloneqq\left \{q_{t} \in \mathbb{R}^d: t \in  
  I\right \},\]
  that pins every element of $\mathcal{T}'$.
 \end{itemize}
\end{theorem}
\begin{proof}
Let $t_1,\dots,t_{|I|}$ be the elements of $I$ in lexicographical order.
For every $0 \le i <|I|$, we define $\mathcal{T}_i\subset \mathcal{T}$, a constant $\delta_i >0$ and a point set
$Q_i = \{q_{t_1},\ldots, q_{t_i}\}$ that satisfy the following.
\begin{itemize}
 \item $|\mathcal{T}_i| \ge \delta_i \binom{n}{k}$;
 \item for every $1 \le j \le i$ and every $\tau \in \mathcal{T}_i$ , the point  $q_{t_j} \in Q_j$ is in the interior of the simplex
 with vertices  $\tau[t_j]$.

\end{itemize}
For $i=0$ this is satisfied with $\mathcal{T}_0\coloneqq\mathcal{T}$, $\delta_0\coloneqq\delta$, and $Q_0\coloneqq\emptyset$. Suppose that we have defined the above for
some $0\le i<|I|$.

Let $\tau \in P^{(d+1)}$; we say that $\tau$ is \emph{heavy with respect to $i$}, if there exists at least
\[\frac{\delta_i}{2}\cdot\binom{k}{d+1}^{-1} \binom{n-(d+1)}{k-(d+1)},\]  $\sigma \in
\mathcal{T}_{i}$ such that
\[\sigma[t_{i+1}]=
\tau.\]
We also say that $\sigma \in \mathcal{T}_{i}$ is  \emph{heavy with respect to $i$}, if $\sigma[t_{i+1}]$
is heavy.

By counting for every subsequence of $P$ of $(d+1)$ points, we have that the number of non-heavy elements of $\mathcal{T}_i$ is
at most
\[\binom{n}{d+1} \cdot \frac{\delta_i}{2}\cdot\binom{k}{d+1}^{-1} \binom{n-(d+1)}{k-(d+1)} = \frac{\delta_i}{2} \binom{n}{k}.\]
Let $\mathcal{T}_i' \subset \mathcal{T}_i$ be the set of heavy elements of $\mathcal{T}_i$.
Therefore, $|\mathcal{T}_i'| \geq \frac{\delta_i}{2} \binom{n}{k}$. Since every subsequence $(d+1)$ points of $P$ can be completed to at most
$\binom{n-(d+1)}{k-(d+1)}$ elements of $\mathcal{T}_i'$, we have that the number of heavy subsequences of $(d+1)$ points of $P$
is at least \[\left. \frac{\delta_i}{2}\cdot \binom{n}{k} \,\middle/\,  \binom{n-(d+1)}{k-(d+1)} \right. = \left. \frac{\delta_i}{2}\cdot \binom{n}{d+1} \,\middle/\,  \binom{k}{d+1} \right. =\frac{\delta_i}{2}\binom{k}{d+1}^{-1} \cdot \binom{n}{d+1}  .\]

Since $\frac{\delta_i}{2}\binom{k}{d+1}^{-1}$ is a constant, independent of $n$, by the second selection lemma there exists a point $q_{t_{i+1}}$ and a constant $\beta_i = \beta_{SS}(d, \frac{\delta_i}{2}\binom{k}{d+1}^{-1})>0$, such that
$q_{t_{i+1}}$ is in at least
\[\beta_i \binom{n}{d+1}\]
simplices, whose vertices are  heavy subsequences of $(d+1)$ points of $P$; let $\mathcal{T}_{i+1} \subset \mathcal{T}_i'$ be the elements $\sigma \in \mathcal{T}_i'$ such that
$q_{t_{i+1}}$ is in the interior of the simplex with vertices $\sigma[t_{i+1}].$
Since, the $\sigma[t_{i+1}]$ are heavy, we have that
\[|\mathcal{T}_{i+1}| \ge \beta_i  \binom{n}{d+1} \cdot  \frac{\delta_i}{2}\binom{k}{d+1}^{-1} \binom{n-(d+1)}{k-(d+1)} =\frac{\delta_i\beta_i}{2} \cdot \binom{n}{k}.\]
Let \[\delta_{i+1}\coloneqq \frac{\delta_i\beta_i}{2}  \textrm{ and } Q_{i+1}\coloneqq Q_i \cup \{q_{t_{i+1}}\}.\]

The result now follows by setting $\mathcal{T}'\coloneqq \mathcal{T}_{|I|}$, $Q\coloneqq Q_{|I|}$ and $\beta\coloneqq \delta_{|I|}$.
\end{proof}

\section{Proof of \cref{thm:main_label}}\label{sec:proof}

For convenience, let us restate \cref{thm:main_label}.

\mainlabel*

\begin{proof}
By \cref{lem:transversals}, it is sufficient to show that there are linear-sized disjoint subsets $X_1, \ldots, X_k$ of $P$ with $\rho$-type transversals, of which a constant fraction are subsequences of $P$.

To this end, we want to apply induction, which requires the following lemma, a strengthening of the theorem.
Here, instead of only a single sequence $P$, we are given $r$ sequences of $n$ points $P_1, \ldots, P_r$.
The lemma then asserts the existence of linear-size subsets of $[n]$, say $X_1, \ldots, X_k$, such that each of the $r$ sequences $P_i$ when restricted to the indices in the $X_j$, that is, $(P_i[X_1], \ldots, P_i[X_k])$ has $\rho_i$-type transversals for some given collection of $\rho_i$.
This version also takes colors into account but instead of admitting individual colors for each $P_i$, the index set $[n]$ is colored only once.
The special case $r=1$ is the statement we need to prove.

\begin{lemma}\label{lem:synced-sets}

    Let $n$ and $(d,\delta,k,\kappa,r)$ be parameters that satisfy $n,d,k,\kappa,r \in \N$, $n \geq k \geq d+1$, $\delta \in \R_+$ and let $(\mathcal{P}, \col, \mathcal{I})$ be a 3-tuple, where
    \begin{itemize}
        \item $\mathcal{P} = \{P_1, \ldots, P_r\}$ is a collection of $r$ sequences on $n$ points in $\R^d$, each of which is in general position,
        \item $\col:[n]\to [\kappa]$ is a $\kappa$-coloring of $[n]$ that induces a coloring $\col(P_i(j))=\col(j)$ of the points in the sequences $P_i$, and
        \item $\mathcal{I} \subseteq [n]^k_<$ is a collection of at least $ \delta \binom{n}{k}$ $k$-element subsequences of $[n]$ such that for any two of them $\tau,\tau' \in \mathcal{I}$, and any $i \in [r]$, $P_i[\tau]$ and $P_i[\tau']$ are copies of one another. That is, all $P_i[\tau]$ are copies of a shared sequence of points $\rho_i$.
    \end{itemize}
    Then, there exists a constant $c' = c'(d,\delta,k,\kappa,r)$ and disjoint subsets $X_1, \ldots, X_k$ of $[n]$ with $|X_i| \geq c' \cdot n$, such that for every transversal $\tau$ of $(X_1,\ldots, X_k)$, and every $i \in [r]$, $(P_i(\tau(1), \ldots, P_i(\tau(k))$ is a copy of $\rho_i$.
    Further, a constant fraction of those transversals are subsequences of $[n]$.
\end{lemma}
We use induction on $d$. Let us start with the inductive step, for which we assume that the statement is proved for any combination of parameters in $\R^{d-1}$.

\paragraph{Induction step $d-1 \to d$:}
For any parameters $(d, \delta, k,\kappa, r)$, there are $\delta' >0$, and $r' \in \N$ such that
$$c'(d, \delta, k,\kappa, r) \geq c'\left(d-1, \delta', k,\kappa,r'\right).$$

Let $m \coloneqq r \cdot \binom{k}{d+1}$ and identify numbers in $[m]$ with tuples in $[r] \times \binom{[k]}{d+1}$.
We iteratively take linear-size subcollections of $\mathcal{I}$
$$\mathcal{I} = \mathcal{I}_0 \supseteq \mathcal{I}_1 \supseteq \ldots \supseteq \mathcal{I}_r \eqqcolon \tilde{\mathcal{I}}$$ in order to produce points $p_{(i,S)}$ where $(i,S) \in [m]$ with the property that
\begin{equation}\label{eq:main_thm_pin_sets}
    p_{(i,S)} \in \conv(P_i[\tau[S]]) \text{ for every }\tau \in \mathcal{I}_{i}.
\end{equation}
To get from $\mathcal{I}_{i-1}$ to $\mathcal{I}_i$, we apply \cref{thm:pin} to 
$$P = P_i,\text{ } I=[k]^{(d+1)}_< \text{, and } \mathcal{T} = \{P_i[\tau] \mid \tau \in \mathcal{I}_{i-1}\}.$$
This yields a collection of points $\{p_{(i,S)}\}_{S \in \binom{[d]}{k+1}}$, each of which pins the family $\mathcal{T'} \subseteq \mathcal{T}$ when restricted to the simplices with indices in $S$.
We may assume that the $p_{(i,S)}$ are in general position with respect to $P_i$ by slight perturbation.
From $\mathcal{T'}$ we can extract the desired subcollection of $k$-tuples
$$\mathcal{I}_i \coloneqq \{\tau \in \mathcal{I}_{i-1} \mid P_i[\tau] \in \mathcal{T}'\}$$
that fulfills \cref{eq:main_thm_pin_sets}.

Note that we can lower bound the size of the final collection by $|\tilde{\mathcal{I}}| \geq \beta_r|\mathcal{I}|$, where $$\beta_0 \coloneqq \delta\text{ and }\beta_i \coloneqq \beta(\beta_{i-1},n,d).$$
Here, $\beta$ is the function defined in \cref{thm:pin} and so $\beta_r$ is a constant.

Now, let for each $(i,S) \in [m]$
$$P'_{(i,S)} \coloneqq \pi_{p_{(i,S)}}(P_i) \subseteq \R^{d-1}$$
denote the point set $P_i$ after a signed central projection from $p_{(i,S)}$ with respect to some suitable hyperplanes.
For any $p \in P'_{(i,S)}$, let $\sgn_{(i,S)}(p)$ denote the sign given to that point by the projection.
Note that by \cref{lem:proj}, each $P_{(i,S)}'$ is in general position.

We will apply the induction hypothesis on the collection of these point sets $\mathcal{P'} \coloneqq \{P'_1\,\ldots,P'_m\}$, a $2^m \kappa$-coloring $\col' : [n] \to [2^m \kappa]$, and a further refinement of $k$-tuples $\mathcal{I}' \subseteq \tilde{\mathcal{I}}$. 
For the coloring, we identify $[2^m\kappa]$ with $[\kappa] \times \{+,-\}^m$ and define
$$\col'(r) \coloneqq (\col(r),\sgn_1(P_1(r)), \ldots, \sgn_m(P_m(r))).$$
For $\mathcal{I}'$, we iteratively pass to subcollections of $\tilde{\mathcal{I}}$
$$\tilde{\mathcal{I}} = \tilde{\mathcal{I}}_0 \supseteq \tilde{\mathcal{I}}_1 \supseteq \ldots \supseteq \tilde{\mathcal{I}}_m \eqqcolon \mathcal{I}'.$$
Here, $\tilde{\mathcal{I}}_j$ is obtained from $\tilde{\mathcal{I}}_{j-1}$ by choosing those $\tau \in \tilde{\mathcal{I}}_{j-1}$ whose sign sequence
$$(\sgn_j(P'_j(\tau(1))), \ldots, \sgn_j(P'_j(\tau(k)))$$
combined with the order type of the sequence
$$\{p_{j}\} \cup P_i[\tau], \text{ where }j=(i,S) \in [m]$$
is most common.
The size of the final collection is then at least $|\mathcal{I}'| \geq (\frac{1}{2^k\cdot f(d,k+1)})^{m}|\tilde{\mathcal{I}}|$, where $f(d,k+1)$ is the number of order types\footnote{This bound could be substantially improved by use of the fact that all relevant order types differ in at most one point, but we choose not do make this explicit since optimizing bounds is not the focus of this paper.} of size $k+1$ in $\R^d$.

Let $\tau^* \in \mathcal{I}'$ be some arbitrary $k$-tuple in the collection.
In order to apply the induction hypothesis to $(\mathcal{P}',\col',\mathcal{I}')$, we need to argue that for every $(i,S) \in [m]$ and $\tau \in \mathcal{I}'$, $P'_{(i,S)}[\tau]$ is a copy of $\rho_{(i,S)}' \coloneqq P'_{(i,S)}[\tau^*]$:
First, the points in $P'_{(i,S)}[\tau]$ and $\rho_i'$ have the same color with respect to $\col'$ since 
\begin{itemize}
    \item $P_i[\tau]$ and $P_i[\tau^*]$ have the same colors with respect to $\col$, and
    \item the signs of the points after projecting with any $\pi_{j}$ are the same, by choice of $\mathcal{I}'$.
\end{itemize} 
Next, to see that $P'_{(i,S)}[\tau]$ and $\rho_{(i,S)}'$ also have the same order type, recall that \cref{lem:proj} implies that the order type of a point set $P'_{(i,S)}[\tau]$ after a signed central projection only depends on the signs of the projected points and the order type of $\{p_{(i,S)}\} \cup P_{i}[\tau]$.
Since these properties are the same for any $\tau \in \mathcal{I}'$ by the choice of $\tilde{\mathcal{I}}_i$, we indeed have $P'_{(i,S)}[\tau] \simeq \rho_{(i,S)}'$.

Applying the induction hypothesis to $(\mathcal{P}',\col',\mathcal{I}')$ gives disjoint sets $X'_1,\ldots,X'_k \subseteq [n]$ of size $c'(d-1,\beta_r\cdot\frac{1}{2^kf(d,k+1)},k,\kappa2^m,m)\cdot n$ with the property that for each $j=(i,S) \in [m]$ $$(P'_{(i,S)}[X_1'], \ldots, P'_{(i,S)}[X_k'])\text{ has } \rho_{(i,S)}'\text{-transversals.}$$

Note that a constant fraction of the transversals to $(X_1', \ldots, X'_k)$ are subsequences of $[n]$. 
We conclude the induction step by arguing that $(P_{i}[X_1'], \ldots, P_{i}[X'_k])$ has $P_i[\tau^*]$-transversals for any $i\in[r]$.
To this end, let $\tau$ be some transversal to $(X_1', \ldots, X'_k)$.
First, the colors of the index sets agree $\col(\tau^*) = \col(\tau)$ by the induction hypothesis and since $\col'$ refines $\col$.
For the order types, it is sufficient by definition to show 
$$P_i[\tau[S]] \simeq P_i[\tau^*[S]]$$ 
for each choice of $d+1$ indices $S \in \binom{[k]}{d+1}$. 
By construction of $\tilde{\mathcal{I}}$, 
$$p_{(i,S)}\in \conv(P_i[\tau^*[S]])$$ 
and by induction hypothesis 
$$\pi_{p_{(i,S)}}(P_i[\tau[S]]) = P'_{(i,S)}[\tau[S]] \simeq P'_{(i,S)}[\tau^*[S]] = \pi_{p_{(i,S)}}(P_i[\tau^*[S]]).$$
Thus, \cref{thm:proj} indeed implies that 
$P_i[\tau[S]] \simeq P_i[\tau^*[S]],$
as desired.
\qed

We are left with proving the base case of the induction.
\paragraph{Case $d=1$:} For any list of parameters $(d=1,\delta, k, \kappa, r)$,
$c'(d, \delta, k,\kappa, r) > 0.$
As in the induction step, we let $m \coloneqq r \cdot \binom{k}{2}$, then identify $[m]$ with tuples in $[r] \times \binom{[k]}{2}$, and pass to a subcollection
$$\tilde{\mathcal{I}} \subseteq \mathcal{I} \text{ with } |\tilde{\mathcal{I}}| \geq \beta_r |\mathcal{I}| \text{ for some } \beta_r > 0$$
by repeated application of \cref{thm:pin}. 
Thus, $\tilde{\mathcal{I}}$ satisfies \cref{eq:main_thm_pin_sets} with respect to a collection of pinning points $\{p_{(i,\{a,b\})}\}_{(i,\{a,b\})\in[m]}$.
Let $\tau^* \in \tilde{\mathcal{I}}$ be some fixed $k$-tuple and for each $i \in [r]$, let $\pi_i$ be a permutation of $[k]$ such that $\pi_i$ sorts the points in $P_i[\tau^*]$ by index, that is,
\begin{equation}\label{eq:main_proof_induction_base}
    P_i[\tau^*](\pi_i(1)) < p_{(i,\{\pi_i(1),\pi_i(2)\})} < P_i[\tau^*](\pi_i(2)) < \ldots < p_{(i,\{\pi_i(k-1),\pi_i(k)\})} < P_i[\tau^*](\pi_i(k)).
\end{equation}
Note that the permutations are independent of the choice of $\tau^*$ since the order that the pinning points separate the points in $P_i[\tau^*]$ is always the same.
Now, the desired index sets that induce $\rho_i$-type transversals of the $P_i$ are simply 
$$X_l \coloneqq \{s \in [n] \mid \tau(l)=s \text{ for some }\tau \in \tilde{\mathcal{I}}\}.$$
The $X_l$ are disjoint since $P_i(X_l)$ and $P_i(X_{l'})$ are separated by $p_{(i,\{l,l'\})}$.
Indeed, by the choice of the $X_i$, each transversal $\tau$ of $(X_1, \ldots, X_k)$ fulfills 
\begin{itemize}
    \item \cref{eq:main_proof_induction_base} for each $i \in [r]$, and
    \item $\col(\tau(j)) = \col(\tau^*(j))$ for each $j \in [k]$,
\end{itemize}
and so 
$$P_i[\tau] \simeq P_i[\tau^*] \simeq \rho_i \text{ for every transversal } \tau \text{ of }(X_1, \ldots, X_k).$$
Further, each $\tau \in \tilde{\mathcal{I}}$ is a subsequence of $[n]$ and a transversal of $(X_1, \ldots, X_k)$. 
Thus, we have for the size of each $X_l$
$$|X_l| \geq \frac{|\tilde{\mathcal{I}}|}{\prod_{j \neq l}|X_j|} \geq \frac{\delta \beta_r \binom{n}{k}}{n^{k-1}} \geq \frac{\delta \beta_r}{k^{k-1}}\cdot n,$$
a constant fraction of $n$.
Finally, since each $\tau \in \tilde{\mathcal{I}}$ is a subsequence of $[n]$, so are a constant fraction of all transversals.
This finishes the proof of the case $d=1$, thus the proof of \cref{lem:synced-sets} and, by using the special case $r=1$, the proof of our main theorem.
\end{proof}

We conclude the paper with an open problem.
\begin{problem}
Let $\rho$ be an (uncolored and unlabelled) order type on $k$ points in $\mathbb{R}^d$, and let $P$ be a set of $n$ points in $\mathbb{R}^d$, such
that $P$ has the maximum number possible of copies of $\rho$. Does there exist disjoint subsets $(X_1,\dots,X_k)$ each with $\approx \frac{n}{k}$ points,
such that $(X_1,\dots,X_k)$ has $\rho$-type transversals?
\end{problem}

\section{Acknowledgments}
This work was initiated at the \emph{15th Crossing Numbers Workshop 2025, Guanajuato, M\'exico}. We thank the participants for many fruitful discussions. In particular, we thank Birgit Vogtenhuber for suggesting that every rectilinear drawing of $K_n$ with few crossings should contain a blow-up of a non-convex $4$-tuple and to Stefan Felsner for pointing out the connection to permutation packing.

{
\small \bibliographystyle{alpha}
\bibliography{blow}
}
\end{document}